\documentclass[12pt,a4paper]{article}
\usepackage{amssymb}
\usepackage{amsmath}
\usepackage{amsthm}
\usepackage{indentfirst}
\usepackage{ot1patch}
\usepackage{psfrag}
\usepackage{rotate}

\newtheorem{tw}{Theorem}[section]
\newtheorem{pr}[tw]{Proposition}
\newtheorem{lm}[tw]{Lemma}
\newtheorem{cor}[tw]{Corollary}

\theoremstyle{definition}
\newtheorem{df}[tw]{Definition}
\newtheorem{ex}[tw]{Example}

\DeclareMathOperator{\jac}{jac}

\DeclareMathOperator{\Irr}{Irr}
\DeclareMathOperator{\Sqf}{Sqf}
\DeclareMathOperator{\Spec}{Spec}
\DeclareMathOperator{\Prime}{Prime}
\DeclareMathOperator{\Gpr}{Gpr}
\DeclareMathOperator{\Rdl}{Rdl}

\newcommand{\rpr}{\mbox{\small\rm rpr}}

\author{Piotr J\k{e}drzejewicz, \L ukasz Matysiak, Janusz Zieli\'nski}
\title{On some factorial properties of subrings}
\date{}

\begin{document}

\maketitle

\begin{abstract}
We discuss various factorial properties of subrings as well as
properties involving irreducible and square-free elements,
in particular ones connected with Jacobian conditions.
\end{abstract}

\begin{table}[b]\footnotesize\hrule\vspace{1mm}
\end{table}

\section*{Introduction}

Throughout the paper by a ring we mean a commutative ring with unity.
By a domain we mean a (commutative) ring without zero divisors.
By $R^{\ast}$ we denote the set of all invertible elements of a ring $R$.
If $R$ is a domain, then by $R_0$ we denote its field of fractions.
If elements $a,b\in R$ are associated in a ring $R$, we write $a\sim_R b$.
We write $a\mid_R b$ if $b$ is divisible by $a$ in $R$.
Furthermore, we write $a\:\rpr_R\,b$ if
$a$ and $b$ are relatively prime in $R$,
that is, have no common non-invertible divisors.
We use a sub-index indicating the ring when we compare properties
in a ring $A$ and in its subring $R$.
If $R$ is a ring, then
by $\Irr R$ we denote the set of all irreducible elements of $R$,
and by $\Sqf R$ we denote the set of all square-free elements of $R$,
where an element $a\in R$ is called square-free if it can not
be presented in the form $a=b^2c$ with $b\in R\setminus R^{\ast}$,
$c\in R$.

\medskip

Now, let $A=k[x_1,\dots,x_n]$ be the algebra of polynomials
over a field $k$ of characteristic zero.
Let $f_1,\dots,f_r\in A$ be algebraically independent over~$k$,
where $r\in\{1,\dots,n\}$.
Put $R=k[f_1,\dots,f_r]$.
By $\jac^{f_1,\dots,f_r}_{x_{i_1},\dots,x_{i_r}}$
denote the Jacobian determinant of $f_1$, $\dots$, $f_r$
with respect to $x_{i_1}$, $\dots$, $x_{i_r}$.
Recall from \cite{analogs} the following generalization of
the Jacobian Conjecture.

\medskip

\noindent
{\bf\boldmath Conjecture.}
\emph{If
$\gcd\big(\jac^{f_1,\dots,f_r}_{x_{i_1},\dots,x_{i_r}},\;
1\leqslant i_1<\ldots<i_r\leqslant n\big)\in k\setminus\{0\}$,
then $R$ is algebraically closed in $A$.}

\medskip

Note that by Nowicki's characterization the assertion means that
$R$ is a ring of constants of some $k$-derivation of $A$
(\cite{Nrings}, Theorem~5.5, \cite{Npolder}, Theorem~4.1.5,
\cite{DaigleBook}, 1.4).
Note also that case $r=1$ is true (\cite{Ayad}, Proposition 14,
see also \cite{closed}, Proposition 4.2).
Moreover, we refer the reader to van den Essen's book
\cite{EssenBook} for information on the Jacobian Conjecture.

\medskip

By a generalization of the results of \cite{keller} and \cite{BondtYan},
the above gcd condition was expressed in terms of irreducible
and square-free elements.

\medskip

\noindent
{\bf Theorem (\cite{analogs}, 2.4)}
\emph{The following conditions are equivalent:}

\smallskip

\noindent
{\rm (i)} \
$\gcd\big(\jac^{f_1,\dots,f_r}_{x_{i_1},\dots,x_{i_r}},\;
1\leqslant i_1<\ldots<i_r\leqslant n\big)\in k\setminus\{0\}$,

\smallskip

\noindent
{\rm (ii)} \
$\Irr R\subset\Sqf A$,

\smallskip

\noindent
{\rm (iii)} \
$\Sqf R\subset\Sqf A$.

\medskip

Finally recall that condition (iii) in much more general case
is equivalent to some factoriality property.

\medskip

\noindent
{\bf Theorem (\cite{analogs}, 3.4)}
\emph{Let $A$ be a unique factorization domain.
Let $R$ be a subring of $A$ such that $R^{\ast}=A^{\ast}$
and $R_0\cap A=R$.
The following conditions are equivalent:}

\smallskip

\noindent
{\rm (i)} \
$\Sqf R\subset\Sqf A$,

\smallskip

\noindent
{\rm (ii)} \
\emph{for every $x\in A$, $y\in\Sqf A$,
if $x^2y\in R\setminus\{0\}$, then $x,y\in R$.}

\medskip

A subring $R$ satisfying the above condition (ii) is called
square-factorially closed in $A$ (\cite{analogs}, Definition 3.5).
Under the assumptions of Theorem 3.4 square-factorially closed
subrings are root closed (\cite{analogs}, Theorem 3.6),
see \cite{Anderson} and \cite{BrewerCMcC} for information
on root closed subrings.

\medskip

The above theorem corresponds with the known fact that
a subring $R$ of a UFD $A$ such that $R^{\ast}=A^{\ast}$
is factorially closed in $A$ if and only if $\Irr R\subset\Irr A$.
Rings of constants of locally nilpotent derivations
in domains of characteristic zero are factorially closed
(see \cite{FreudenburgBook} and \cite{DaigleBook} for details).

\medskip

A general question stated in \cite{analogs}, when do conditions
(ii) or (iii) of Theorem 2.4 imply algebraic closedness of $R$ in $A$,
inspired us to study their relations with other conditions
of this type (see Proposition \ref{33} below).
The above Theorem 3.4 motivates us to investigate various properties
having a form of factoriality, in particular similar to (ii).

\section{Divisibility, relative primeness, etc.\ in a subring}
\label{s1}

In this section we describe relationships between various conditions
on a subring of a domain of arbitrary characteristic.

\begin{pr}
\label{11}
Let $A$ be a domain, let $R$ be a subring of $A$.
The following conditions are equivalent:

\medskip

\noindent
{\rm (i)} \
$R_0\cap A=R$,

\medskip

\noindent
{\rm (ii)} \
$R=L\cap A$ for some subfield $L\subset A_0$,

\medskip

\noindent
{\rm (iii)} \
for every $a\in R$, $b\in A$,
if $ab\in R\setminus\{0\}$, then $b\in R$,

\medskip

\noindent
{\rm (iv)} \
for every $a,b\in R$, if $a\mid_A b$, then $a\mid_R b$.
\end{pr}

\begin{proof}
$\text{\rm (i)}\Leftrightarrow \text{\rm (iv)}$ \
Obviously, $R\subset R_0\cap A$.
Thus the equation $R_0\cap A=R$ is equivalent to the inclusion
$R_0\cap A\subset R$, the latter means that
for arbitrary $a,b\in R$, $a\neq 0$,
if $\frac{b}{a}\in A$, then $\frac{b}{a}\in R$,
what is the statement (iv).

\medskip

\noindent
$\text{\rm (i)}\Rightarrow \text{\rm (ii)}$ \
Put $L=R_0$.

\medskip

\noindent
$\text{\rm (ii)}\Rightarrow \text{\rm (i)}$ \
If $R=L\cap A$, then $R\subset L$,
and since $L$ is a field, we have $R_0\subset L$.
Hence, $R_0\cap A\subset L\cap A=R$.
The opposite inclusion is evident.

\medskip

\noindent
$\text{\rm (iii)}\Rightarrow \text{\rm (iv)}$ \
Let $a,b\in R$.
If $b=0$, then obviously $a\mid_R b$.
Let $b\neq 0$.
If $a\mid_A b$, then $b=ac$ for some $c\in A$.
By (iii) we have $c\in R$, and consequently $a\mid_R b$.

\medskip

\noindent
$\text{\rm (iv)}\Rightarrow \text{\rm (iii)}$ \
If $ab\in R\setminus\{0\}$ for $a\in R, b\in A$, then
$a\mid_A ab$.
Therefore $a\mid_R ab$, by (iv).
Hence $ab=ar$ for some $r\in R$.
If $a\neq 0$, then $b=r\in R$, since $A$ is a domain.
If $a=0$, then $ab=0$, contrary to $ab\in R\setminus\{0\}$.
\end{proof}

\begin{pr}
\label{12}
Let $A$ be a domain, let $R$ be a subring of $A$.
The following conditions are equivalent:

\medskip

\noindent
{\rm (i)} \
$R^{\ast}=A^{\ast}\cap R$,

\medskip

\noindent
{\rm (ii)} \
for every $a\in R$, if $a\in A^{\ast}$, then $a\in R^{\ast}$,

\medskip

\noindent
{\rm (iii)} \
for every $a,b\in R$, if $a\:\rpr_A\,b$, then $a\:\rpr_R\,b$.
\end{pr}

\begin{proof}
$\text{\rm (i)}\Leftrightarrow \text{\rm (ii)}$ \
Clearly $R^{\ast}\subset A^{\ast}\cap R$.
Therefore the equality $R^{\ast}=A^{\ast}\cap R$
is equivalent to the inclusion $A^{\ast}\cap R\subset R^{\ast}$,
the latter is a formulation of condition (ii).

\medskip

\noindent
$\text{\rm (ii)}\Rightarrow \text{\rm (iii)}$ \
Assume that (ii) holds and consider elements $a,b\in R$
relatively prime in $A$.
If $c$ is a common divisor of $a$ and $b$ in $R$,
then it is obviously their common divisor in $A$.
Hence $c$ is invertible in $A$,
then by (ii) it is invertible in $R$.
Consequently, $a$ and $b$ are relatively prime in $R$.

\medskip

\noindent
$\text{\rm (iii)}\Rightarrow \text{\rm (ii)}$ \
It is sufficient to notice that
$a\in R$ is invertible (in $A$ or $R$, respectively)
if and only if it is relatively prime with $1$.
\end{proof}

\begin{pr}
\label{13}
Let $A$ be a domain.
Let $R$ be a subring of $A$.
Consider the following conditions:

\medskip

\noindent
{\rm (i)} \
for every $a,b\in R$, if $a\mid_A b$, then $a\mid_R b$,

\medskip

\noindent
{\rm (ii)} \
for every $a,b\in R$, if $a\sim_A b$, then $a\sim_R b$,

\medskip

\noindent
{\rm (iii)} \
for every $a,b\in R$, if $a\:\rpr_A\,b$, then $a\:\rpr_R\,b$.

\medskip

\noindent
{\rm (iv)} \
for every $a\in R$, if $a\in \Irr A$, then $a\in \Irr R$.

\medskip
\noindent
Then we have:
$$\text{\rm (i)}\Rightarrow \text{\rm (ii)}\Rightarrow
\text{\rm (iii)}\Rightarrow \text{\rm (iv)}.$$
\end{pr}

\begin{proof}
\noindent
$\text{\rm (i)}\Rightarrow \text{\rm (ii)}$ \
It suffices to note that $a,b\in R$ are associated
(in $A$ or $R$, respectively) if and only if
$a\mid b$ and~$b\mid a$.

\medskip

\noindent
$\text{\rm (ii)}\Rightarrow \text{\rm (iii)}$ \
It suffices to observe that $a\in R$ is invertible
(in $A$ or $R$, respectively) if and only if
it is associated with $1$.
Then the assertion follows from the equivalence
${\rm (ii)}\Leftrightarrow {\rm (iii)}$ of Proposition \ref{12}.

\medskip

\noindent
$\text{\rm (iii)}\Rightarrow \text{\rm (iv)}$ \
Assume that (iii) holds and consider $a\in R$ reducible in $R$.
Then $a=bc$ for some elements $b,c\in R$ not invertible in $R$.
From (iii) we deduce that $b$ and $c$ are not invertible in~$A$
(see Proposition \ref{12}), hence $a$ is reducible in $A$.
\end{proof}

As a consequence of Propositions \ref{11}, \ref{12}, \ref{13}
we obtain Corollary \ref{14}.

\begin{cor}
\label{14}
If $A$ is a domain and $R$ is a subring of $A$,
then the following implications hold:
\renewcommand{\arraystretch}{0.2}
$$\begin{array}{rcccc}
R_0\cap A=R &&&& \\
 & \mbox{\begin{psfrags}\rotatebox{-25}{$\Rightarrow$}\end{psfrags}}
 &&& \\
 && A^{\ast}\cap R=R^{\ast} & \Rightarrow & R\cap\Irr A\subset\Irr R. \\
 & \mbox{\begin{psfrags}\rotatebox{25}{$\Rightarrow$}\end{psfrags}}
 &&& \\
R^{\ast}=A^{\ast} &&&&
\end{array}$$
\end{cor}

It is easily seen that none of the one-way implications
of Proposition \ref{13} can be reversed in general.

\begin{ex}
\label{15}
If $A$ is equal to $k[x]$, a polynomial ring
in one variable over a field $k$,
and $R=k[x^2,x^3]$, then condition (i) is not fulfilled,
because $x^3$ is divisible by $x^2$ in $A$,
but it is not in $R$.
However (ii) holds, because $A^{\ast}=k\setminus\{0\}$,
hence if $f$ and $g$ are associated in $A$,
then $f=cg$ for some $c\in k\setminus\{0\}$,
that is, $f$ and $g$ are associated in $R$.
\end{ex}

\begin{ex}
\label{16}
If $A=k(x)[y]$ and $R=k[xy,y]$, where $k$ is a field,
then condition (ii) does not hold, because $xy$ and $y$
are associated in $A$, yet they are not in $R$.
Condition (iii) is fulfilled, since
$A^{\ast}\cap R=k(x)\cap k[xy,y]=k=R^{\ast}$
(see Proposition \ref{12}).
\end{ex}

\begin{ex}
\label{17}
If $A$ is a field and $R$ is not a field, then (iii)
is not fulfilled, because an irreducible element of $R$
is not relatively prime with~$1$ in~$R$,
but is relatively prime with~$1$ in~$A$.
Condition (iv) holds, since $\Irr A=\emptyset $.
\end{ex}

It is worth noting that the conditions in Corollary \ref{14}
have no immediate relationship with unique factorization in $R$.
More precisely, the strongest conditions in Corollary \ref{14}
do not imply unique factorization in $R$,
and unique factorization in $R$ does not imply the weakest
of the conditions in question.
Here are examples.

\begin{ex}
\label{18}
Let $A=k[x,y]$ and $R=k[x^2,y^2,xy]$, where $k$ is a field.
Then both conditions $R^{\ast}=A^{\ast}$ and $R_0\cap A=R$ are
fulfilled, but there is no unique factorization in $R$,
since $x^2\cdot y^2=(xy)^2$.
\end{ex}

\begin{ex}
\label{19}
Let $A=k(x)[y]$ and $R=k[x,y]$, where $k$ is a field.
Clearly $R$ is a unique factorization domain,
however $R\cap\Irr A\subset\Irr R$ does not hold, because $xy$
is irreducible in $A$ and reducible in $R$.
\end{ex}

\section{Factoriality with respect to a subring}
\label{s2}

We introduce the notion of factorial closedness of one subring
with respect to factors from another subring.

\begin{df}
\label{21}
Let $B$ be a subring of $A$.
The subring $R$ of $A$ is called $B$-factorially closed,
if, whenever $a\in A$, $b\in B$ and $ab\in R\setminus\{0\}$,
then $a\in R$.
If $R$ is $R$-factorially closed, then we call it
self-factorially closed.
\end{df}

Note that \emph{$A$-factorially closed} in the sense
of the above definition is equivalent to usual notion
of \emph{factorially closed} (in $A$).

\begin{pr}
\label{22}
Let $A$ be a domain of characteristic $p>0$
and let $R$ be a subring of $A$ such that $A^p\subset R$,
where $A^p=\{a^p,\,a\in A\}$.
The following conditions are equivalent:

\smallskip

\noindent
{\rm (i)} \
the ring $R$ is separably algebraically closed in $A$,

\smallskip

\noindent
{\rm (ii)} \
$R_0\cap A=R$,

\smallskip

\noindent
{\rm (iii)} \
the ring $R$ is self-factorially closed in $A$,

\smallskip

\noindent
{\rm (iv)} \
the ring $R$ is $A^p$-factorially closed in $A$.
\end{pr}

\begin{proof}
$\text{(i)}\Leftrightarrow\text{(ii)}$
was stated in \cite{rings}, Proposition 2.2,

\medskip

\noindent
$\text{(ii)}\Leftrightarrow\text{(iii)}$
follows from Lemma~\ref{11},

\medskip

\noindent
$\text{(iii)}\Rightarrow\text{(iv)}$
is obvious.

\medskip

\noindent
$\text{(iv)}\Rightarrow\text{(iii)}$
Assume that condition (iv) holds and consider
$a\in A$ and $b\in R$ such that $ab\in R\setminus\{0\}$.
Then $ab^p\in R\setminus\{0\}$ so, by the assumption, $a\in R$.
\end{proof}

If $R$ is a finitely generated $K$-algebra such that $A^p\subset R$,
then the above equivalent conditions characterize $R$
as a ring of constants of some $K$-deri\-va\-tion of $A$.

\section{A general diagram of implications}
\label{s3}

In this section we consider various properties similar to
$\Irr R\subset\Sqf A$ and $\Sqf R\subset\Sqf A$,
and we present basic relations between them.

\medskip

Given a ring $R$, we denote the following sets:
\begin{itemize}
\item[--]
$\Prime R$ of all prime elements of $R$,
\item[--]
$\Gpr R$ of (single) generators of principal radical
ideals of $R$,
\item[--]
$\Rdl R$ of radical ideals of $R$
(see \cite{BelluceDiNF}, p.\ 68).
\end{itemize}

\begin{lm}
\label{31}
If $R$ is a ring, then:

\medskip

\noindent
{\bf a)}
$\Irr R\subset\Sqf R$,

\medskip

\noindent
{\bf b)}
$\Prime R\subset\Gpr R$.
\end{lm}

\begin{proof}
{\bf a)}
Consider an element $x\in R$.
Assume that $x\notin\Sqf R$, that is $x=y^2z$,
where $y\in R\setminus R^{\ast}, z\in R$.
Then $x=y\cdot (yz)$, where $y,yz\in R\setminus R^{\ast}$,
so $x\notin\Irr R$.

\medskip

\noindent
{\bf b)}
This holds because every prime ideal is radical.
\end{proof}

\begin{lm}
\label{32}
If $R$ is a domain, then:

\medskip

\noindent
{\bf a)}
$\Prime R\subset\Irr R$,

\medskip

\noindent
{\bf b)}
$\Gpr R\subset\Sqf R$.
\end{lm}

\begin{proof}
{\bf a)}
This fact is well known.

\medskip

\noindent
{\bf b)}
Consider an element $r\in\Gpr R$.
Let $r=x^2y$, where $x,y\in R$.
We have $(xy)^2=ry$, so $(xy)^2\in Rr$, and then $xy\in Rr$,
because $Rr$ is a radical ideal.
We obtain $xy=rz$, so $xy=x^2yz$, and hence $1=xz$.
\end{proof}

From the above lemmas we obtain.

\begin{pr}
\label{33}
Let $A$ be a domain.
Let $R$ be a subring of~$A$.
The following implications hold:
\renewcommand{\arraycolsep}{0.5mm}
\begin{small}
$$\begin{array}{ccccccc}
\Irr R\subset\Irr A & \Rightarrow & \Prime R\subset\Irr A &
\Leftarrow & \Prime R\subset\Prime A & \Leftarrow &
\forall_{I\in\Spec R} AI\in\Spec A \\
\Downarrow & & \Downarrow & & \Downarrow & & \Downarrow \\
\Irr R\subset\Sqf A & \Rightarrow & \Prime R\subset\Sqf A &
\Leftarrow & \Prime R\subset\Gpr A & \Leftarrow &
\forall_{I\in\Spec R} AI\in\Rdl A \\
\Uparrow & & \Uparrow & & \Uparrow & & \Uparrow \\
\Sqf R\subset\Sqf A & \Rightarrow & \Gpr R\subset\Sqf A &
\Leftarrow & \Gpr R\subset\Gpr A & \Leftarrow &
\forall_{I\in\Rdl R} AI\in\Rdl A
\end{array}$$
\end{small}
\end{pr}

\section{Some factorial conditions for subrings}
\label{s4}

The last section contains various properties in a factorial form.
In the first proposition we express a condition from \cite{analogs},
Theorem~3.4 in terms of irreducible and square-free factorizations.

\begin{pr}
\label{41}
Let $A$ be a UFD and
let $R$ be a subring of~$A$ such that $R^{\ast}=A^{\ast}$.
The following conditions are equivalent:

\medskip

\noindent
{\rm (i)} \
$\forall\, a\in A, \; \forall\, b\in\Sqf A, \;
a^2b\in R\setminus\{0\}\Rightarrow a,b\in R$,

\medskip

\noindent
{\rm (ii)} \
$\forall\, s_0,\dots,s_n\in\Sqf A, \;
s_n^{2^n}\dots s_1^2s_0\in R\Rightarrow s_0,\dots , s_n\in R$,

\medskip

\noindent
{\rm (iii)} \
$\forall\, q_1,\dots,q_n\in\Irr A, \, q_i\not\sim_A q_j, \,
i\neq j, \; \forall\, k_1,\dots,k_n\geqslant 0,$
$$q_1^{k_1}\dots q_n^{k_n}\in R\Rightarrow
\forall\, i, \; q_1^{c^{(1)}_i}\dots q_n^{c^{(n)}_i}\in R,$$
where $k_j=c^{(j)}_r2^r+\ldots+c^{(j)}_02^0$ for $j=1,\dots,n$,
and $c^{(j)}_i\in\{0,1\}$ for $i=0,\dots,r$.
\end{pr}

\begin{proof}
$\text{\rm (ii)}\Rightarrow\text{\rm (i)}$ \
Assume {\rm (ii)} and consider elements $a\in A$, $b\in\Sqf  A$.
Let $a=s_n^{2^n}\ldots s_1^2s_0$, where $s_0,\ldots,s_n\in\Sqf A$.
If $a^2b=s_n^{2^{n+1}}\ldots s_1^{2^2}s_0^2b\in R\setminus\{0\}$,
then $s_n,\ldots,s_1,s_0,b\in R$ by {\rm (ii)},
and then $a=s_n^2\ldots s_1^2s_0\in R$.

\medskip

\noindent
$\text{\rm (iii)}\Rightarrow\text{\rm (ii)}$ \
Assume condition {\rm (iii)}.
Let $s_0,\ldots,s_n\in\Sqf A$ satisfy $s_n^{2^n}\ldots s_1^2s_0\in R$.
We can write $s_i=u_iq_1^{c^{(1)}_{i}}\ldots q_m^{c^{(m)}_{i}}$,
where $u_i\in A^{\ast}$, $q_1,\ldots,q_m\in\Irr A$
with $q_j\not\sim_A q_l$ for $j\neq l$ and $c^{(j)}_{i}\in\{0, 1\}$.
Then $s_n^{2^n}\ldots s_1^2s_0=u_n^{2^n}\ldots
u_1^2u_0\cdot q_1^{k_1}\ldots q_m^{k_m}$,
where $k_j=c^{(j)}_{n}2^n+\ldots+c^{(j)}_{1}2+c^{(j)}_{0}$.
By the assumption, if $q_1^{k_1}\ldots q_m^{k_m}\in R$,
then $q_1^{c^{(1)}_{i}}\ldots q_m^{c^{(m)}_{i}}\in R$
for $i=1,2,\ldots,n$.
Hence $s_0, \ldots, s_n\in R$.

\medskip

\noindent
$\text{\rm (ii)}\Rightarrow\text{\rm (iii)}$ \
Assume {\rm (ii)}.
Let $q_1^{k_1}\ldots q_n^{k_n}\in R$, where
$q_1,\ldots,q_n\in\Irr A$, $q_j\not\sim_A q_l$ for $j\neq l$.
Put $k_j=c^{(j)}_{r}2^r+\ldots+c^{(j)}_{1}2+c^{(j)}_{0}$
for $j=1,\ldots,n$, where $c^{(j)}_{i}\in\{0,1\}$.
Let $s_i=q_1^{c^{(1)}_{i}}\ldots q_n^{c^{(n)}_{i}}$.
By the assumption, since $s_n^{2^n}\ldots s_1^2s_0\in R$,
we obtain $s_0,\ldots,s_n\in R$.

\medskip

\noindent
$\text{\rm (i)}\Rightarrow\text{\rm (ii)}$ \
Simple induction.
\end{proof}

Note that factorial closedness can be expressed in the following way.

\begin{pr}
\label{42}
Let $A$ be a UFD and
let $R$ be a subring of~$A$ such that $R^{\ast}=A^{\ast}$.
The following conditions are equivalent:

\medskip

\noindent
{\rm (i)} \
$\forall\, a,b\in A, \;
ab\in R\setminus\{0\}\Rightarrow a,b\in R$,

\medskip

\noindent
{\rm (ii)} \
$\forall\, q_1,\dots, q_n\in\Irr A, \;
\forall\, k_1,\dots, k_n\geqslant 1, \;
q_1^{k_1}\dots q_n^{k_n}\in R\Rightarrow q_1,\dots,q_n\in R$.
\end{pr}


In the next two propositions we consider factorizations
with respect to relatively prime elements.

\begin{pr}
\label{43}
Let $A$ be a UFD and
let $R$ be a subring of~$A$ such that $R^{\ast}=A^{\ast}$.
The following conditions are equivalent:

\medskip

\noindent
{\rm (i)} \
$\forall\, a,b\in A, \, a\:\rpr\,b, \;
ab\in R\setminus\{0\}\Rightarrow a,b\in R$,

\medskip

\noindent
{\rm (ii)} \
$\forall\, a_1,\ldots,a_n\in A, \, a_i\:\rpr\,a_j, \, i\neq j, \;
a_1\ldots a_n\in R\setminus\{0\}\Rightarrow a_1,\ldots,a_n\in R$,

\medskip

\noindent
{\rm (iii)} \
$\forall\, q_1,\dots,q_n\in\Irr A, \, q_i\not\sim_A q_j, \,
i\neq j, \;\forall\, k_1,\dots,k_n\geqslant 1,$
$$q_1^{k_1}\ldots q_n^{k_n}\in R\Rightarrow
q_1^{k_1},\ldots,q_n^{k_n}\in R.$$
\end{pr}

\begin{proof}
We see that {\rm (i)} and {\rm (iii)} are the special cases
of {\rm (ii)}.

\medskip

\noindent
$\text{\rm (i)}\Rightarrow\text{\rm (ii)}$ \
Simple induction.

\medskip

\noindent
$\text{\rm (iii)}\Rightarrow\text{\rm (i)}$ \
Assume {\rm (iii)}.
Consider $a,b\in A$, $a\:\rpr\,b$, such that $ab\in R\setminus\{0\}$.
Put $a=uq_1^{k_1}\ldots q_s^{k_s}$ and
$b=vq_{s+1}^{k_{s+1}}\ldots q_n^{k_n}$, where $u,v\in A^{\ast}$,
$q_1,\dots,q_n\in\Irr A$, $q_i\not\sim_A q_j$ for $i\neq j$.
By the assumption, since
$ab=uvq_1^{k_1}\ldots q_s^{k_s}q_{s+1}^{k_{s+1}}\ldots q_n^{k_n}\in R$,
we have
$q_1^{k_1},\ldots,q_s^{k_s},q_{s+1}^{k_{s+1}},\ldots,q_n^{k_n}\in R$.
Finally, $a=uq_1^{k_1}\ldots q_s^{k_s}\in R$ and
$b=vq_{s+1}^{k_{s+1}}\ldots q_n^{k_n}\in R$.
\end{proof}

\begin{pr}
\label{44}
Let $A$ be a UFD and
let $R$ be a subring of~$A$ such that $R^{\ast}=A^{\ast}$.
The following conditions are equivalent:

\medskip

\noindent
{\rm (i)} \
$\forall\, a,b\in A, \, a\:\rpr\,b, \; \forall\, k>1, \;
(a^kb\in R\setminus\{0\}\Rightarrow a,b\in R)$,

\medskip

\noindent
{\rm (ii)} \
$\forall\, q_1,\dots,q_n\in\Irr A, \, q_i\not\sim_A q_j, \,
i\neq j, \; \forall k_1,\dots,k_r>1, \, r\leqslant n,$
$$q_1^{k_1}\dots q_r^{k_r}q_{r+1}\dots q_n\in R\Rightarrow
q_1,\dots,q_r,q_{r+1}\ldots q_n\in R.$$
\end{pr}

\begin{proof}
$\text{\rm (i)}\Rightarrow\text{\rm (ii)}$ \
Assume {\rm (i)}.
Let $q_1^{k_1}\dots q_r^{k_r}q_{r+1}\dots q_n\in R$
for some pairwise non-associated $q_1,\dots,q_n\in\Irr A$
and for some $k_1,\dots,k_r>1$.
By the assumption we have
$q_1\in R$ and $q_2^{k_2}\dots q_r^{k_r}q_{r+1}\dots q_n \in R$,
then $q_2\in R$ and $q_3^{k_3}\dots q_r^{k_r}q_{r+1}\dots q_n \in R$,
and so on, until we obtain $q_r\in R$ and $q_{r+1}\ldots q_n\in R$.

\medskip

\noindent
$\text{\rm (ii)}\Rightarrow\text{\rm (i)}$ \
Assume {\rm (ii)}.
Consider $a,b\in A$ such that $a\:\rpr\,b$
and $a^kb\in R\setminus\{0\}$ for some $k>1$.
We can write $a=uq_1^{l_1}\dots q_r^{l_r}$
and $b=vq_{r+1}^{l_{r+1}}\dots q_s^{l_s}q_{s+1}\dots q_n$,
where $u,v\in A^{\ast}$, $q_1,\dots,q_n\in\Irr A$,
$q_i\not\sim_A q_j$ for $i\neq j$ and $l_{r+1},\dots,l_s>1$.
We have:
$a^kb=u^kv(q_1^{l_1}\dots q_r^{l_r})^k
q_{r+1}^{l_{r+1}}\dots q_s^{l_s}q_{s+1}\dots q_n \in R$.
Then $q_1,q_2,\dots,q_s,$ $q_{s+1}\ldots q_n\in R$.
Finally, $a,b\in R$.
\end{proof}

The following proposition shows that if we omit the restriction
$b\in\Sqf R$ in Proposotion~\ref{41},
then we obtain usual factorial closedness.

\begin{pr}
\label{45}
Let $A$ be a domain.
Let $R$ be a subring of~$A$.
The following conditions are equivalent:

\medskip

\noindent
{\rm (i)} \
$\forall\, a,b\in A, \;
ab\in R\setminus\{0\}\Rightarrow a,b\in R$,

\medskip

\noindent
{\rm (ii)} \
$\forall\, a,b\in A, \;
a^2b\in R\setminus\{0\}\Rightarrow a,b\in R$,

\medskip

\noindent
{\rm (iii)} \
$\forall\, a,b\in A, \;
a^3b\in R\setminus\{0\}\Rightarrow a,b\in R$,

\medskip

\noindent
{\rm (iv)} \
$\forall\, a,b\in A, \;
(a^2b\in R\setminus\{0\}\vee a^3b\in R\setminus\{0\})
\Rightarrow a,b\in R$,

\medskip

\noindent
{\rm (v)} \
$\forall\, a,b\in A, \; \forall\, k>1 \;
(a^kb\in R\setminus\{0\}\Rightarrow a,b\in R)$,

\medskip

\noindent
{\rm (vi)} \
$\forall\, a,b\in A, \; \forall\, k\geqslant 1 \;
(a^kb\in R\setminus\{0\}\Rightarrow a,b\in R)$.
\end{pr}

\begin{proof}
$\text{\rm (ii)}\Rightarrow\text{\rm (i)}$ \
Assume that {\rm (ii)} holds.
Let $ab\in R\setminus \{0\}$ for some $a,b\in A$.
Then $a^2b^2\in R\setminus \{0\}$.
By the assumption we have $a,b^2\in R$.
Again using assumption for $b^2\cdot 1\in R$ we have $b\in R$.
Finally, $a,b\in R$.

\medskip

\noindent
$\text{\rm (iii)}\Rightarrow\text{\rm (i)}$ \
Assume that {\rm (iii)} holds.
Let $ab\in R\setminus \{0\}$ for some $a,b\in A$.
Then $a^3b^3\in R\setminus \{0\}$.
By the assumption we have $a,b^3\in R$.
Again using assumption for $b^3\cdot 1$ we have $b\in R$.
Finally, $a,b\in R$.

\medskip

\noindent
Implications:
$\text{\rm (vi)}\Rightarrow\text{\rm (v)}$,
$\text{\rm (v)}\Rightarrow\text{\rm (iv)}$,
$\text{\rm (iv)}\Rightarrow\text{\rm (ii)}$,
$\text{\rm (iv)}\Rightarrow\text{\rm (iii)}$
are obvious.

\medskip

\noindent
$\text{\rm (i)}\Rightarrow\text{\rm (vi)}$ \
Simple induction.
\end{proof}

The last proposition is motivated by a modification
of Proposition~\ref{45}.

\begin{pr}
\label{46}
Let $A$ be a domain.
Let $R$ be a subring of~$A$.
The following conditions are equivalent:

\medskip

\noindent
{\rm (i)} \
$\forall\, a,b\in A, \;
ab,a^2b\in R\setminus\{0\}\Rightarrow a,b\in R$,

\medskip

\noindent
{\rm (ii)} \
$\forall\, a,b\in A, \;
a^2b,a^3b\in R\setminus\{0\}\Rightarrow a,b\in R$,

\medskip

\noindent
{\rm (iii)} \
$\forall\, a,b\in A \; (\forall\, k\geqslant 1 \;
a^kb\in R\setminus\{0\})\Rightarrow a,b\in R$,

\medskip

\noindent
{\rm (iv)} \
$\forall\, a,b\in A \; (\forall k>1 \;
a^kb\in R\setminus\{0\})\Rightarrow a,b\in R$,

\medskip

\noindent
{\rm (v)} \
$\forall\, a,b\in A \;
(\exists\, k_0, \; \forall\, k\geqslant k_0, \;
a^kb\in R\setminus\{0\})\Rightarrow a,b\in R$.
\end{pr}

\begin{proof}
$\text{\rm (iii)}\Rightarrow\text{\rm (i)},\text{\rm (ii)},
\text{\rm (v)}$ \
Assume that condition (iii) holds.
It is enough to prove that if $a^kb, a^{k+1}b\in R\setminus\{0\}$
for some $a,b\in A$, then $a^{k+2}b\in R\setminus\{0\}$
and (if $k>1$) $a^{k-1}b\in R\setminus\{0\}$.

\medskip

Assume that $a^kb, a^{k+1}b\in R\setminus\{0\}$, where $a,b\in A$.
Since $$(a^kb)^l\cdot a^{k+2}b=(a^{k+1}b)^2\cdot (a^kb)^{l-1}\in R$$
holds for every $l\geqslant 1$, from the assumption we obtain
$a^{k+2}b\in R$.
Moreover, if $k>1$, since
$$(a^{k+1}b)^l\cdot a^{k-1}b=(a^kb)^2\cdot (a^{k+1}b)^{l-1}\in R$$
also holds for every $l\geqslant 1$, we infer also $a^{k-1}b\in R$.

\medskip

\noindent
Implications: $\text{\rm (v)}\Rightarrow\text{\rm (iv)}$,
$\text{\rm (iv)}\Rightarrow\text{\rm (iii)}$,
$\text{\rm (i)}\Rightarrow\text{\rm (iii)}$,
$\text{\rm (ii)}\Rightarrow\text{\rm (iv)}$ are obvious.
\end{proof}

\end{document}